\documentclass[12pt,a4paper]{article}
\usepackage{amssymb,amsmath}
\usepackage{mathrsfs}
\usepackage{graphicx}
\usepackage{amsthm}
\theoremstyle{definition} 
\newtheorem{df}{Definition}[section] 
   
\theoremstyle{plain}            
\newtheorem{pro}[df]{Proposition}
\newtheorem{lem}[df]{Lemma}
\newtheorem{theo}[df]{Theorem}
\newtheorem{cor}[df]{Corollary}

\newcommand{\f}{\ensuremath{\varphi}}

\newcommand{\Bscr}{\ensuremath{\mathscr{B}}}

\newcommand{\Hscr}{\ensuremath{\mathscr{H}}}

\newcommand{\Mcal}{\ensuremath{\mathcal{M}}}

\newcommand{\nn}{\ensuremath{\mathbb{N}}}

\newcommand{\unit}{\ensuremath{\mathbf{1}}}

\newcommand{\Ca}{$C${\rm*}-algebra}      
\newcommand{\Csa}{$C${\rm*}-subalgebra}

\newcommand{\vNa}{von Neumann algebra}

\newcommand{\ifff}{if and only if}

\begin{document}

\begin{center}
{\Large \bf Star order and topologies on \vNa{}s}\\
{\large Martin Bohata\footnote{bohata@math.feld.cvut.cz}\\}
         \it Department of Mathematics, Faculty of Electrical Engineering,\\
        Czech Technical University in Prague, Technick\'a 2,\\ 
        166 27 Prague 6, Czech Republic        
\end{center}

{\small \textbf{Abstract:} 
The goal of the paper is to study a topology generated by the star order on von Neumann algebras. In particular, it is proved that the order topology under investigation is finer than $\sigma$-strong* topology. On the other hand, we show that it is comparable with the norm topology if and only if the von Neumann algebra is finite-dimensional.} 


%%%%%%%%%%%%%%%%%%%%%%%%%%%%%%%%%%%%%%%%%%%%%%%%%%%%%%%%%%%%%%%%%%%%%%%%
%%%%%%%%%%%%%%%%%%%%%%%%%%%%%%%%%%%%%%%%%%%%%%%%%%%%%%%%%%%%%%%%%%%%%%%%
\section{Introduction}

In the order-theoretical setting, the notion of convergence of a net was introduced by G. Birkhoff \cite{Bi35,Bi67}. Let $(P,\leq)$ be a poset and let $x\in P$. If $(x_\alpha)_{\alpha\in\Gamma}$ is an increasing net in $(P,\leq)$ with the supremum $x$, we write $x_\alpha \uparrow x$. Similarly, $x_\alpha\downarrow x$ means that $(x_\alpha)_{\alpha\in\Gamma}$ is a decreasing net in $(P,\leq)$ with the infimum $x$. We say that a net $(x_\alpha)_{\alpha\in\Gamma}$ is {\it order convergent to $x$} in $(P,\leq)$ if there are nets $(y_\alpha)_{\alpha\in\Gamma}$ and $(z_\alpha)_{\alpha\in\Gamma}$ in $(P,\leq)$ such that $y_\alpha\leq x_\alpha\leq z_\alpha$ for all $\alpha\in\Gamma$, $y_\alpha\uparrow x$, and $z_\alpha\downarrow x$. If $(x_\alpha)_{\alpha\in\Gamma}$ is order convergent to $x$, we write $x_\alpha\stackrel{o}{\to}x$. It is easy to see that every net is order convergent to at most one point.

The order convergence determines a natural topology on a poset $(P,\leq)$ as follows. A subset $C$ of $P$ is said to be {\it order closed} if no net in $C$ is order convergent to a point in $P\setminus C$. The topology on a poset is called {\it order topology} if the family of all closed sets coincides with the family of all order closed sets. We shall denote the order topology of a poset $(P,\leq)$ by the symbol $\tau_o(P,\leq)$. It is easy to see that the order topology is the finest topology preserving order convergence (i.e. if $\tau$ is a topology on $(P,\leq)$ such that $x_\alpha\stackrel{o}{\to} x$ implies $x_\alpha\stackrel{\tau}{\to} x$, then $\tau\subseteq\tau_o(P,\leq)$). Since every one-point set is closed in $\tau_o(P,\leq)$, the topological space $(P,\tau_o(P,\leq))$ is $T_1$-space.

There are a number of papers dealing with the order topology, in particular on lattices. Lattices with the property that the order convergence coincides with the convergence in the order topology were studied, for example, in \cite{ER95,Gi76}. It was shown in \cite{FK54} that a normed linear space is reflexive \ifff{} the lattice of all its closed linear subspaces is Hausdorff (in the corresponding order topology). This interesting result has a direct consequence that the order topology is not, in general, Hausdorff. 

The order topology on the complete lattice of all projections on a Hilbert space was investigated in \cite{BChW12,Pa95}. A great progress in understanding of the order topologies on projection lattice and self-adjoint part of a von Neumann algebra (endowed with the standard order) was done in \cite{ChHW15}. It was shown that there is a strong connection between these topologies and locally convex topologies on von Neumann algebras. Motivated by this research, we shall study the order topology on various subsets of a von Neumann algebra endowed with the star order.

The rest of the paper is organized as follows. In the second section, we collect some basic facts on von Neumann algebras, star order, order convergence, and order topology. The third section deals with the existence of the suprema and infima in several subsets of a von Neumann algebra with respect to the star order. Moreover, we examine a relationship between suprema and infima of monotone nets and the strong operator limit of these nets. In the last section, we prove that if a net $(x_\alpha)_{\alpha\in\Gamma}$ order converges (with respect to the star order) to $x$, then it also converges to $x$ in $\sigma$-strong* topology. Thus the order topology is finer than $\sigma$-strong* topology. This result seems to be surprising because the star order is not translation invariant and so the order topology is far from being linear. Moreover, we show that the order topology is not comparable with norm topology unless the von Neumann algebra is finite-dimensional. Among other things, we also prove that, for every von Neumann algebra, the restriction $\tau_o(\Mcal_{sa},\preceq)|_{P(\Mcal)}$ of the order topology on self-adjoint part of a \vNa{} \Mcal{} to projection lattice coincides with the order topology $\tau_o(P(\Mcal),\preceq)$ on the projection lattice. This is in the contrast with the case of the order topology with respect to the standard order. It was shown in \cite[Proposition~2.9]{ChHW15} that $\tau_o(\Mcal_{sa},\leq)|_{P(\Mcal)}=\tau_o(P(\Mcal),\leq)$ if and only if the \vNa{} \Mcal{} is abelian.

%%%%%%%%%%%%%%%%%%%%%%%%%%%%%%%%%%%%%%%%%%%%%%%%%%%%%%%%%%%%%%%%%%%%%%%%
%%%%%%%%%%%%%%%%%%%%%%%%%%%%%%%%%%%%%%%%%%%%%%%%%%%%%%%%%%%%%%%%%%%%%%%%
\section{Preliminaries}

We say that a poset $(P,\leq)$ is {\it Dedekind complete} if every nonempty subset of $P$ that is bounded above has the supremum. A poset $(P,\leq)$ is Dedekind complete \ifff{} every nonempty subset of $P$ that is bounded below has the infimum. In the following lemma and proposition, we summarize the well known facts about the order convergence and order topology. We prove these results for convenience of the reader.

\begin{lem}\label{limsup}
Let $(P,\leq)$ be a poset. Assume that $(x_\alpha)_{\alpha\in\Gamma}$ is a net in $P$ and $x\in P$.
	\begin{enumerate}
		\item If $\alpha_0\in\Gamma$ is an arbitrary fixed element, $\Lambda=\{\alpha\in\Gamma|\,\alpha_0\leq \alpha\}$, and $(x_\alpha)_{\alpha\in\Gamma}$ is order convergent to $x$ in $(P,\leq)$, then $(x_\alpha)_{\alpha\in\Lambda}$ is (order) bounded and order convergent to $x$ in $(P,\leq)$.
		\item If $\liminf_\alpha x_\alpha=\limsup_\alpha x_\alpha=x$, then $(x_\alpha)_{\alpha\in\Gamma}$ is order convergent to $x$ in $(P,\leq)$.
		\item If $(P,\leq)$ is Dedekind complete and $(x_\alpha)_{\alpha\in\Gamma}$ is (order) bounded and order convergent to $x$ in $(P,\leq)$, 	then $\liminf_\alpha x_\alpha=\limsup_\alpha x_\alpha=x$.
	\end{enumerate}
\end{lem}
	\begin{proof}\hfill{}
			\begin{enumerate}
					\item Suppose that $\alpha_0\in\Gamma$ is an arbitrary fixed element and $\Lambda=\{\alpha\in\Gamma|\,\alpha_0\leq \alpha\}$. If $(x_\alpha)_{\alpha\in\Gamma}$ is order convergent to $x$ in $(P,\leq)$, then there are nets $(y_\alpha)_{\alpha\in\Gamma}$ and $(z_\alpha)_{\alpha\in\Gamma}$ such that $y_\alpha\leq x_\alpha\leq z_\alpha$ for all $\alpha\in\Gamma$, $y_\alpha\uparrow x$, and $z_\alpha\downarrow x$. Hence $y_\alpha\leq x_\alpha\leq z_\alpha$ for all $\alpha\in\Lambda$. Moreover, since $u\in P$ is an upper bound of $(y_\alpha)_{\alpha\in\Gamma}$ \ifff{} $u$ is an upper bound of $(y_\alpha)_{\alpha\in\Lambda}$, we see that the net $(y_\alpha)_{\alpha\in\Lambda}$ satisfies $y_\alpha\uparrow x$. Similarly, we prove that the net $(z_\alpha)_{\alpha\in\Lambda}$ satisfies $z_\alpha\downarrow x$. Therefore, the net $(x_\alpha)_{\alpha\in\Lambda}$ is order convergent to $x$ in $(P,\leq)$. Since $y_{\alpha_0}\leq y_\alpha\leq x_\alpha\leq z_\alpha\leq z_{\alpha_0}$ for all $\alpha\in\Lambda$, the net $(x_\alpha)_{\alpha\in\Lambda}$ is bounded.
				
					\item If $\liminf_\alpha x_\alpha=\limsup_\alpha x_\alpha=x$, then we set $z_\alpha=\sup_{\alpha\leq\beta} x_\beta$ and $y_\alpha=\inf_{\alpha\leq\beta} x_\beta$ for all $\alpha\in\Gamma$. It is obvious that $y_\alpha\leq x_\alpha\leq z_\alpha$ for all $\alpha\in\Gamma$, $y_\alpha\uparrow x$, and $z_\alpha\downarrow x$ which shows that $x_\alpha\stackrel{o}{\to} x$. 
					
					\item If $x_\alpha\stackrel{o}{\to} x$, then $y_\alpha\leq x_\alpha\leq z_\alpha$ for all $\alpha\in\Gamma$, $y_\alpha\uparrow x$, and $z_\alpha\downarrow x$. We observe that $\inf_{\alpha\leq\beta}x_\beta$ and $\sup_{\alpha\leq\beta}x_\beta$ exist for all $\alpha\in\Gamma$ because $(P,\leq)$ is Dedekind complete. By the boundedness of $(x_\alpha)_{\alpha\in\Gamma}$, the nets $(\sup_{\alpha\leq\beta} x_\beta)_{\alpha\in\Gamma}$ and $(\inf_{\alpha\leq\beta} x_\beta)_{\alpha\in\Gamma}$ are bounded. The Dedekind completeness of $(P,\leq)$ ensures that $\sup_{\alpha\in\Gamma}\inf_{\alpha\leq\beta} x_\beta$ and $\inf_{\alpha\in\Gamma}\sup_{\alpha\leq\beta} x_\beta$ exist. As 
					$$
					y_\alpha\leq\inf_{\alpha\leq\beta} x_\beta\leq x_\alpha\leq\sup_{\alpha\leq\beta} x_\beta\leq z_\alpha
					$$ 
					for all $\alpha\in\Gamma$, we have 
					$$
					x=\sup_{\alpha\in\Gamma}y_\alpha\leq \sup_{\alpha\in\Gamma}\inf_{\alpha\leq\beta} x_\beta\leq \inf_{\alpha\in\Gamma}\sup_{\alpha\leq\beta} x_\beta
						\leq\inf_{\alpha\in\Gamma} z_\alpha=x.
					$$
					This means that $\liminf_\alpha x_\alpha=\limsup_\alpha x_\alpha=x$.
			\end{enumerate}
	\end{proof}

\begin{pro}[{\cite[Proposition 2.3]{ChHW15}}]\label{Dedekind}
Let $(P,\leq)$ be a Dedekind complete poset and let $P_0\subseteq P$ be closed in $\tau_o(P,\leq)$. If the supremum of every nonempty subset of $P_0$ with an upper bound in $P$ belongs to $P_0$, then $\tau_o(P,\leq)|_{P_0}=\tau_o(P_0,\leq)$.
\end{pro}
\begin{proof}
Let $M\subseteq P_0$. Since $M$ is closed in $\tau_o(P,\leq)|_{P_0}$ \ifff{} $M$ is closed in $\tau_o(P,\leq)$, it is sufficient to show that $M$ is closed in $\tau_o(P,\leq)$ \ifff{} $M$ is closed in $\tau_o(P_0,\leq)$.

Let $M$ be closed in $\tau_o(P,\leq)$ and let $(x_\alpha)_{\alpha\in\Gamma}$ be a net in $M$ order converging to $x\in P_0$ in $(P_0,\leq)$. Then there are nets $(y_\alpha)_{\alpha\in\Gamma}$ and $(z_\alpha)_{\alpha\in\Gamma}$ in $(P_0,\leq)$ such that $y_\alpha\leq x_\alpha\leq z_\alpha$ for all $\alpha\in\Gamma$, $y_\alpha\uparrow x$, and $z_\alpha\downarrow x$ (where the supremum of $(y_\alpha)_{\alpha\in\Gamma}$ and the infimum of $(z_\alpha)_{\alpha\in\Gamma}$ are taken in $(P_0,\leq)$). Because $x$ is an upper bound of $(y_\alpha)_{\alpha\in\Gamma}$, $\sup_{\alpha\in\Gamma} y_\alpha$ exists in $(P,\leq)$ and belongs to $P_0$. Hence $\sup_{\alpha\in\Gamma} y_\alpha=x$ in $(P,\leq)$. Similarly, $\inf_{\alpha\in\Gamma} z_\alpha=x$ in $(P,\leq)$. Therefore, $(x_\alpha)_{\alpha\in\Gamma}$ is order convergent to $x$ in $(P,\leq)$. As $M$ is closed in $\tau_o(P,\leq)$, $x\in M$.

Conversely, let $M$ be closed in $\tau_o(P_0,\leq)$ and let $(x_\alpha)_{\alpha\in\Gamma}$ be a net in $M$ order converging to $x\in P$ in $(P,\leq)$. Without loss of generality, we can assume that $(x_\alpha)_{\alpha\in\Gamma}$ is bounded (see Lemma~\ref{limsup}) in $(P,\leq)$. By Lemma~\ref{limsup}, $x=\liminf_\alpha x_\alpha=\limsup_\alpha x_\alpha$. Using the boundedness of $(x_\alpha)_{\alpha\in\Gamma}$, $x\in P_0$. It follows from Lemma~\ref{limsup} that $(x_\alpha)_{\alpha\in\Gamma}$ is order convergent to $x$ in $(P_0,\leq)$. As $M$ is closed in $\tau_o(P_0,\leq)$, $x\in M$.
\end{proof}

The \Ca{} $\Bscr(\Hscr)$ of all bounded operators on a complex Hilbert space \Hscr{} is rich on the interesting topologies. One of them is the {\it strong (operator) topology} which is a locally convex topology on $\Bscr(\Hscr)$ generated by semi-norms 
$$
p_\xi:x\mapsto \|x\xi\|,\quad \xi\in\Hscr, x\in\Bscr(\Hscr).
$$ 
Another topology is the {\it strong* (operator) topology} which is a locally convex topology on $\Bscr(\Hscr)$ generated by semi-norms 
$$
p_\xi:x\mapsto \sqrt{\|x\xi\|^2+\|x^*\xi\|^2}, \quad\xi\in\Hscr, x\in\Bscr(\Hscr).
$$
We denote the strong topology and strong* topology by $\tau_{s}$ and $\tau_{s^*}$, respectively. By a {\it \vNa{}} we shall mean a strongly closed \Csa{} of the \Ca{} $\Bscr(\Hscr)$. Every \vNa{} $\Mcal$ has the predual $\Mcal_*$ which consists of normal linear functionals in $\Mcal^*$. Using the predual, one can define the $\sigma$-strong* topology $s^*(\Mcal,\Mcal_*)$ by the family of semi-norms
$$
p_\f:x\mapsto \sqrt{\f(x^*x)+\f(xx^*)}, \quad\f\in\Mcal_* \mbox{ is positive}.
$$
There are the following relationships between topologies on \Mcal{}:
$$
\tau_s|_\Mcal\subseteq\tau_{s^*}|_\Mcal\subseteq s^*(\Mcal,\Mcal_*)\subseteq\tau_u(\Mcal),
$$
where $\tau_u(\Mcal{})$ denotes the norm topology on a \vNa{} \Mcal{}. Moreover, $\tau_{s^*}$ and $s^*(\Mcal,\Mcal_*)$ concide on every norm bounded subset of \Mcal{}.

Let $x$ and $y$ be elements of a von Neumann algebra \Mcal{}. We write $x\preceq y$ if $x^*x=x^*y$ and $xx^*=yx^*$. The binary relation $\preceq$ on \Mcal{} is a partial order called {\it star order}. Elements $x$ and $y$ are said to be {\it *-orthogonal} if $x^*y=yx^*=0$. A simple observation shows \cite{Bo11} that $x\preceq y$ \ifff{} there is $z\in \Mcal$ such that $x$ and $z$ are *-orthogonal and $y=x+z$. Thus the star order can be regarded as a partial order induced by orthogonality. It was pointed out in \cite{Dr78} that there is a connection of the star order with the Moore-Penrose inverse. The star order is also a natural partial order on partial isometries (see, for example, \cite{Ex17,HL63}).

By $l(x)$ we denote the {\it left support} of $x$ which is the smallest projection $p\in \Mcal$ satisfying $px=x$. The left support of $x$ is the projection onto the closure of the range of $x$ and so it is sometimes called the {\it range projection} of $x$. It is well known that a von Neumann algebra contains the left supports of all its elements. The set of all projections in \Mcal{} is denoted by $P(\Mcal{})$. It forms a complete lattice under the standard order $\leq$ called {\it projection lattice} of \Mcal. We denote the projection lattice simply by the symbol $P(\Mcal)$ (instead of using a more correct symbol $(P(\Mcal),\leq)$). Recall that the standard order $\leq$ coincides with the star order $\preceq$ on $P(\Mcal)$. The self-adjoint part of \Mcal{}, the positive part of \Mcal{}, the set of all invertible elements in \Mcal, and the set of all partial isometries in \Mcal{} are denoted by $\Mcal_{sa}$, $\Mcal_+$, $\Mcal_{inv}$, and $\Mcal_{pi}$, respectively.

\begin{lem}\label{smaller element}
Let \Mcal{} be a \vNa{} and let $x\in \Mcal$. If $y\in \Mcal_+$ (resp. $y\in \Mcal_{pi}$) and $x\preceq y$, then $x\in \Mcal_+$ (resp. $x\in \Mcal_{pi}$).
\end{lem}
	\begin{proof}
		It was proved in \cite[Corollary 2.9]{ACMS10} and \cite[Proposition 3.1]{Bo11}.
	\end{proof}

The previous lemma is no longer true for self-adjoint operators. Indeed, it was pointed out in \cite{BHLL04} that 
$$
\begin{pmatrix}0&1\\0&0\end{pmatrix}\preceq\begin{pmatrix}0&1\\1&0\end{pmatrix}.
$$

%%%%%%%%%%%%%%%%%%%%%%%%%%%%%%%%%%%%%%%%%%%%%%%%%%%%%%%%%%%%%%%%%%%%%%%%
%%%%%%%%%%%%%%%%%%%%%%%%%%%%%%%%%%%%%%%%%%%%%%%%%%%%%%%%%%%%%%%%%%%%%%%%
\section{Infimum and supremum}

Let us recall a useful result proved in \cite{ACMS10}.

 \begin{pro}[{\cite[Theorem 2.7]{ACMS10}}]\label{antezana}
Let $x,y\in\Bscr(\Hscr)$. Then $x\preceq y$ \ifff{} $x=l(x) y$, $l(x)\leq l(y)$, and $l(x)$ commutes with $yy^*$.
 \end{pro}

Let us note that we can omit the condition $l(x)\leq l(y)$ in the previous proposition. Indeed, if $x=l(x) y$ and $l(x)$ commutes with $yy^*$, then $x^*x=x^*l(x)y=(l(x)x)^*y=x^*y$ and $xx^*=l(x)yy^*l(x)=yy^*l(x)=y(l(x)y)^*=yx^*$.

The following proposition is a special case of Theorem 4.4 in \cite{Ci15} (see also \cite[Theorem 7]{Ja83}). Because the proof was omitted in \cite{Ci15}, we prove this result for convenience of the reader.

	\begin{pro}\label{suprema and infima}
Let $M$ be a nonempty subset of a \vNa{} \Mcal{} and let $y\in\Mcal$ be an upper bound of $M$ (with respect to the star order).
			\begin{enumerate}
				\item $\left(\sup_{x\in M}l(x)\right)y$, where $\sup_{x\in M}l(x)$ is considered in $P(\Mcal)$, is the supremum of $M$ in $(\Mcal,\preceq)$.
				\item $\left(\inf_{x\in M}l(x)\right)y$, where $\inf_{x\in M}l(x)$ is considered in $P(\Mcal)$, is the infimum of $M$ in $(\Mcal,\preceq)$.
			\end{enumerate}
	\end{pro}
		\begin{proof}
		\hfill{}
				\begin{enumerate}
						\item Let $p$ be the supremum of $\{l(x)|\,x\in M\}$ in $P(\Mcal)$ and let $y$ be an upper bound of $M$. It is easy to verify that $py$ is an upper bound of $M$.
						
Let $u\in \Mcal$ be an upper bound of $M$. We have to show that $py\preceq u$. Applying Proposition~\ref{antezana}, we see that, for all $x\in M$, $l(x)\leq l(u)$ and $l(x)$ commutes with $uu^*$. Hence $p\leq l(u)$ and $p$ commutes with $uu^*$. Moreover, $l(pu)u=pu$ because $l(pu)=p$. By Proposition~\ref{antezana}, $pu\preceq u$. As $l(x)(y-u)=0$ for all $x\in M$, we have $l(x)l(y-u)=0$ for all $x\in M$ and so $pl(y-u)=0$. It follows from this that $p(y-u)=pl(y-u)(y-u)=0$. Therefore, $py=pu\preceq u$.
						
						\item Let $p$ be the infimum of $\{l(x)|\,x\in M\}$ in $P(\Mcal)$ and let $y$ be an upper bound of $M$. It follows from Proposition~\ref{antezana} that, for each $x\in M$, $x=l(x)y$ and $yy^*$ commutes with $l(x)$. Moreover, $p$ commutes with $yy^*$ because $p$ is an element of the \vNa{} $\{yy^*\}'$. Therefore, 
						\begin{eqnarray*}
							xx^*p&=&l(x)yy^*l(x)p=yy^*l(x)p=yy^*p=pyy^*=pl(x)yy^*\\&=&pl(x)yy^*l(x)=pxx^*
						\end{eqnarray*}
holds for all $x\in M$. By Proposition~\ref{antezana}, we obtain that $py$ is a lower bound of $M$.

If $u\in\Mcal$ is a lower bound of $M$, then $l(u)\leq p$. Since $u\preceq y$, $u=l(u)y=l(u)py$ and $l(u)$ commutes with $yy^*$. Furthermore, $l(u)\leq p$ ensures that $l(u)$ commutes with $p$. Hence $l(u)$ commutes with $pyy^*p=py(py)^*$. Applying Proposition~\ref{antezana}, $u\preceq py$.
				\end{enumerate}
		\end{proof}

Let us note that if $M$ is an empty subset of a \vNa{} \Mcal{}, then the supremum of $M$ in $(\Mcal,\preceq)$ is 0 and the infimum of $M$ in $(\Mcal,\preceq)$ does not exist. 

The statement (iii) in the following corollary is easily seen from \cite[Theorem~4.4]{Ci15} and the fact that bounded (with respect to the star order) set of self-adjoint elements has a self-adjoint upper bound (for this, see the proof of the statement).

		\begin{cor}\label{supinf1}
Let \Mcal{} be a \vNa{}. Then the following statements hold:
			\begin{enumerate}
				  \item The poset $(\Mcal,\preceq)$ is Dedekind complete.
					\item The supremum of every subset of $P(\Mcal)$ in $(\Mcal,\preceq)$ is a projection. The infimum of every nonempty subset of $P(\Mcal)$ in $(\Mcal,\preceq)$ is a projection.
					\item The supremum of every bounded set $M\subseteq \Mcal_{sa}$ in $(\Mcal,\preceq)$ is a self-adjoint element. The infimum of every nonempty set $M\subseteq \Mcal_{sa}$ in $(\Mcal,\preceq)$ is a self-adjoint element.
					\item The supremum of every bounded set $M\subseteq \Mcal_+$ in $(\Mcal,\preceq)$ is a positive element. The infimum of every nonempty set $M\subseteq \Mcal_+$ in $(\Mcal,\preceq)$ is a positive element.
			\end{enumerate}
		\end{cor}
				\begin{proof}\hfill{}
					\begin{enumerate}
							\item The statement follows directly from Proposition~\ref{suprema and infima}.
							
							\item It is clear that $\mathbf{1}\in P(\Mcal)$ is an upper bound of every subset $M$ of $P(\Mcal)$. If $M\subseteq P(\Mcal)$ is nonempty, then Proposition~\ref{suprema and infima} implies that the supremum and the infimum of $M$ in $(\Mcal,\preceq)$ are projections. Moreover, the supremum of the empty set in $(\Mcal,\preceq)$ is equal to the infimum of \Mcal{} which is 0.
							
							\item Let $M\subseteq \Mcal_{sa}$ be a nonempty and let $y\in\Mcal$ be an upper bound of $M$. It is easy to see that $y^*$ is also upper bound of $M$. It follows from \cite[Proposition~2.4]{Bo11} that $u=\frac{y+y^*}{2}$ is an upper bound of $M$. According to Proposition~\ref{suprema and infima}, $s=\left(\sup_{x\in M} l(x)\right)u$ is the supremum of $M$. Since $x=l(x) u$ for each $x\in M$, $l(x)$ commutes with $u$ for every $x\in M$. Thus $\left(\sup_{x\in M} l(x)\right)\in \{u\}'$ and so $\left(\sup_{x\in M} l(x)\right)$ commutes with $u$. Therefore, $s=\left(\sup_{x\in M} l(x)\right)u$ is self-adjoint. If $M$ is empty, then the supremum of $M$ is 0.
							
			Let $M$ be a nonempty subset of $\Mcal_{sa}$ and let 
			$$L_M=\{u\in\Mcal|\,u\preceq x \mbox{ for all }x\in M\}.$$
			The set $L_M$ is nonempty and bounded above. Therefore, $L_M$ has the supremum $s$ of the form $s=\left(\sup_{x\in L_M}l(x)\right)y$, where $y\in M$ is an arbitrary fixed element. Let us show that $s$ is self-adjoint. Obviously, $s\in L_M$. As $M$ is a set of self-adjoint elements and the involution preserves the star order, we have $s^*\in L_M$ which gives $s^*\preceq s$. It follows from this that $s\preceq s^*$, and therefore $s=s^*$.
							
							\item Since $x\preceq y$ implies $|x|\preceq |y|$ (see \cite[Corollary~2.13]{ACMS10} or \cite[Corollary~2.9]{Bo11}), we can assume without loss of generality that an upper bound $u$ of the nonempty set $M\subseteq \Mcal_+$ is positive. 
							%Since $\sup_{x\in M} l(x)$ and $u$ are positive mutually commuting elements
							According to Lemma~\ref{smaller element}, $s=\left(\sup_{x\in M} l(x)\right)u$ is positive. If $M$ is empty, then the supremum of $M$ is 0 in $(\Mcal,\preceq)$.
							
Let $M$ be a nonempty subset of $\Mcal_+$ and let 
							$$
							L_M=\{u\in\Mcal|\,u\preceq x \mbox{ for all }x\in M\}.
							$$
The set $L_M$ is nonempty and bounded above by a positive element. Therefore, $L_M$ contains only positive elements (see Lemma~\ref{smaller element}). Since $\inf_{x\in M} x=\sup_{x\in L_M} x$, $\inf_{x\in M} x$ has to be positive.
							
					\end{enumerate}			
				\end{proof}

It follows directly from the previous corollary that posets $(\Mcal_{sa},\preceq)$ and $(\Mcal_+,\preceq)$ are Dedekind complete. Furthermore, if $M$ is a bounded subset of $\Mcal_{sa}$ (resp. $\Mcal_+$), then the supremum of $M$ in $(\Mcal_{sa},\preceq)$ (resp. $(\Mcal_+,\preceq)$) coincides with the supremum of $M$ in $(\Mcal,\preceq)$. Similarly, we have the equality of the infima of $M$ in $(\Mcal_{sa},\preceq)$ (resp. $(\Mcal_+,\preceq)$) and in $(\Mcal,\preceq)$ whenever $M$ is nonempty subset of $\Mcal_{sa}$ (resp. $\Mcal_+$).

In the same spirit as before, we can prove that the supremum and the infimum of a set of partial isometries are again partial isometries. The case of the supremum can also be found in \cite[Theorem 12]{Ja83}.	

\begin{cor}\label{supinf2}
Let $\Mcal_{pi}$ be the set of all partial isometries in a von Neumann algebra \Mcal. The supremum of every bounded subset of $\Mcal_{pi}$ in $(\Mcal,\preceq)$ is a partial isometry. The infimum of every nonempty subset of $\Mcal_{pi}$ in $(\Mcal,\preceq)$ is a partial isometry.
\end{cor}
\begin{proof}
Let $M\subseteq \Mcal_{pi}$ be bounded and nonempty. By \cite[Theorem 2.15]{ACMS10}, there is a partial isometry $u$ such that it is an upper bound of $M$. Set $p=\sup_{x\in M} l(x)$. It follows from Proposition~\ref{suprema and infima} that $pu$ is the supremum of $M$. 
%We see from Proposition~\ref{antezana} that $uu^*$ commutes with $p$. Furthermore, $u$ is a partial isometry and so $u=uu^*u$. This gives 
%$$
%pu(pu)^*(pu)=puu^*pu=puu^*u=pu.
%$$
By Lemma~\ref{smaller element}, we see that $pu$ is a partial isometry. If $M$ is empty, then the supremum of $M$ is 0 in $(\Mcal,\preceq)$.
							
				Let $M$ be a nonempty subset of $\Mcal_{pi}$ and let 
				$$
				L_M=\{u\in\Mcal|\,u\preceq x \mbox{ for all }x\in M\}.
				$$
The set $L_M$ is nonempty and bounded above by a partial isometry. Using Lemma~\ref{smaller element}, we obtain that $L_M$ contains only partial isometries. Since $\inf_{x\in M} x=\sup_{x\in L_M} x$, $\inf_{x\in M} x$ has to be a partial isometry.
\end{proof}

The strong operator limit of monotone nets in $(\Bscr(\Hscr),\preceq)$ was studied in \cite{ACMS10}. Furthermore, a connection between suprema of increasing nets in $(\Bscr(\Hscr)_{sa},\preceq)$ and the strong operator limit was shown in \cite{Gu06,PV07}. We prove a similar result to that of \cite[Theorem 4.5]{PV07}.

	\begin{theo}\label{monotone nets}
Let \Mcal{} be a \vNa{}.
		\begin{enumerate}
				\item If $(x_\alpha)_{\alpha\in\Gamma}$ is an increasing net in $(\Mcal,\preceq)$ and bounded above, then the strong (operator) limit of $(x_\alpha)_{\alpha\in\Gamma}$ exists and it is equal to the supremum of $(x_\alpha)_{\alpha\in\Gamma}$.
				\item If $(x_\alpha)_{\alpha\in\Gamma}$ is a decreasing net in $(\Mcal,\preceq)$, then the strong (operator) limit of $(x_\alpha)_{\alpha\in\Gamma}$ exists and is equal to the infimum of $(x_\alpha)_{\alpha\in\Gamma}$.
		\end{enumerate}
	\end{theo}
		\begin{proof}
		\hfill{}
				\begin{enumerate}
						\item By Proposition~\ref{antezana}, $(l(x_\alpha))_{\alpha\in\Gamma}$ is an increasing net of projections and so it has the strong limit, say $p$, which is the supremum of $(l(x_\alpha))_{\alpha\in\Gamma}$ in $P(\Mcal)$ (see \cite[Proposition 2.5.6]{Ka97I}). Let $y$ be an upper bound of $(x_\alpha)_{\alpha\in\Gamma}$. We infer from Proposition~\ref{suprema and infima} that the supremum of $(x_\alpha)_{\alpha\in\Gamma}$ is $p y$. Applying Proposition~\ref{antezana}, $x_\alpha=l(x_\alpha)y$ for all $\alpha\in\Gamma$. Since multiplication is separately continuous in the strong (operator) topology, we see that the net $(l(x_\alpha) y)_{\alpha\in\Gamma}=(x_\alpha)_{\alpha\in\Gamma}$ is strongly convergent to $py$.
						
						\item We can assume without loss of generality that $(x_\alpha)_{\alpha\in\Gamma}$ is bounded above. If $(x_\alpha)_{\alpha\in\Gamma}$ is not bounded above, we take an fixed element $\alpha_0\in\Gamma$ and consider $(x_\alpha)_{\alpha\in\Lambda}$, where $\Lambda=\{\alpha\in\Gamma|\,\alpha_0\leq\alpha\}$. The net $(x_\alpha)_{\alpha\in\Lambda}$ is bounded above by $x_{\alpha_0}$ because $(x_\alpha)_{\alpha\in\Gamma}$ is decreasing. It is easy to see that $(x_\alpha)_{\alpha\in\Lambda}$ has the same set of all lower bounds as the net $(x_\alpha)_{\alpha\in\Gamma}$. Moreover, $(x_\alpha)_{\alpha\in\Gamma}$ is strongly convergent to $x$ \ifff{} $(x_\alpha)_{\alpha\in\Lambda}$ is strongly convergent to $x$.
						
						The following discussion is analogous to that of the proof of (i). By Proposition~\ref{antezana}, $(l(x_\alpha))_{\alpha\in\Gamma}$ is a decreasing net of projections and so it has the strong limit, say $p$, which is the infimum of $(l(x_\alpha))_{\alpha\in\Gamma}$ in $P(\Mcal)$ (see \cite[Corollary~2.5.7]{Ka97I}). Let $y$ be an upper bound of $(x_\alpha)_{\alpha\in\Gamma}$. We infer from Proposition~\ref{suprema and infima} that the infimum of $(x_\alpha)_{\alpha\in\Gamma}$ is $p y$. Applying Proposition~\ref{antezana}, $x_\alpha=l(x_\alpha)y$ for all $\alpha\in\Gamma$. Since multiplication is separately continuous in the strong (operator) topology, we see that the net $(l(x_\alpha) y)_{\alpha\in\Gamma}=(x_\alpha)_{\alpha\in\Gamma}$ is strongly convergent to $py$.
				\end{enumerate}
		\end{proof}

%%%%%%%%%%%%%%%%%%%%%%%%%%%%%%%%%%%%%%%%%%%%%%%%%%%%%%%%%%%%%%%%%%%%%%%%
%%%%%%%%%%%%%%%%%%%%%%%%%%%%%%%%%%%%%%%%%%%%%%%%%%%%%%%%%%%%%%%%%%%%%%%%
\section{Comparison of topologies}

\begin{lem}\label{norm inequality}
Let $x$ and $y$ be elements of a \vNa{} \Mcal. If $x\preceq y$, then $\|x\|\leq\|y\|$.
\end{lem}
		\begin{proof}
If $x\preceq y$, then $x^*x=x^*y$. Thus $\|x\|^2=\|x^*y\|\leq \|x^*\|\|y\|=\|x\|\|y\|$. It follows from this that $\|x\|\leq\|y\|$.
		\end{proof}

The previous lemma shows that every bounded subset of a \vNa{} with respect to the star order is necessarily norm bounded. The converse is clearly not true because, for example, the set $\{\unit, 2\unit\}$ is norm bounded but it is not bounded above with respect to the star order.

We have seen that there is a close relationship between strong topology and (star) order convergence. This motivates the question whether the relative topology $\tau_s|_{\Mcal}$ on a \vNa{} \Mcal{} is comparable with the order topology $\tau_o(\Mcal,\preceq)$.

\begin{pro}\label{sot}
Let $(x_\alpha)_{\alpha\in\Gamma}$ be a net in a \vNa{} \Mcal{} and let $x\in\Mcal$. If $x_\alpha\stackrel{o}{\to} x$ in $(\Mcal,\preceq)$, then $x_\alpha\stackrel{\tau_s}{\to} x$. In particular, $\tau_s|_{\Mcal}\subseteq \tau_o(\Mcal,\preceq)$.
\end{pro}
		\begin{proof}
Let $(x_\alpha)_{\alpha\in\Gamma}$ be a net in \Mcal{} such that $x_\alpha\stackrel{o}{\to} x$ in $(\Mcal,\preceq)$. Then there are nets $(y_\alpha)_{\alpha\in\Gamma}$ and $(z_{\alpha})_{\alpha\in\Gamma}$ in $(\Mcal,\preceq)$ such that $y_\alpha\preceq x_\alpha\preceq z_\alpha$ for all $\alpha\in \Gamma$, $y_\alpha\uparrow x$, and $z_\alpha\downarrow x$. Let $\alpha_0$ be an fixed element of $\Gamma$ and let $\Lambda=\{\alpha\in\Gamma|\,\alpha_0\leq\alpha\}$. To investigate strong convergence of $(x_\alpha)_{\alpha\in\Gamma}$ it is sufficient to consider the net $(x_\alpha)_{\alpha\in\Lambda}$ in place of $(x_\alpha)_{\alpha\in\Gamma}$. Because $(y_\alpha)_{\alpha\in\Lambda}$ is increasing and bounded above by $x$ in $(\Mcal,\preceq)$ and $(z_\alpha)_{\alpha\in\Lambda}$ is decreasing in $(\Mcal,\preceq)$, we obtain from Theorem~\ref{monotone nets} that $y_\alpha\stackrel{\tau_s}{\to} x$ and $z_\alpha\stackrel{\tau_s}{\to} x$. Let $\xi$ be an element of the underlying Hilbert space. Clearly,
		$$\|x_\alpha\xi-x\xi\|=\|x_\alpha\xi-y_\alpha\xi+y_\alpha\xi-x\xi\|\leq\|x_\alpha\xi-y_\alpha\xi\|+\|y_\alpha\xi-x\xi\|.$$
Since $y_\alpha\stackrel{\tau_s}{\to} x$, it is sufficient to prove that $\|x_\alpha\xi-y_\alpha\xi\|\to 0$. One can easily verify that $y_\alpha\preceq x_\alpha$ implies $x_\alpha-y_\alpha\preceq x_\alpha$ and so $x_\alpha-y_\alpha\preceq z_\alpha$. Hence $(x_\alpha-y_\alpha)^*(x_\alpha-y_\alpha)=(x_\alpha-y_\alpha)^*z_\alpha$. By this and Cauchy-Schwarz inequality,
\begin{eqnarray*}
		\|x_\alpha\xi-y_\alpha\xi\|^2 &=& \left\langle x_\alpha\xi-y_\alpha\xi,x_\alpha\xi-y_\alpha\xi\right\rangle=
																			\left\langle (x_\alpha-y_\alpha)^*(x_\alpha-y_\alpha)\xi,\xi\right\rangle\\
																	&\leq& \|(x_\alpha-y_\alpha)^*(x_\alpha-y_\alpha)\xi\|\|\xi\|
																	 =\|(x_\alpha-y_\alpha)^*z_\alpha\xi\|\|\xi\|\\
																	&=&\|(x_\alpha-y_\alpha)^*(z_\alpha-y_\alpha+y_\alpha)\xi\|\|\xi\|\\
																	&=&\|(x_\alpha-y_\alpha)^*(z_\alpha-y_\alpha)\xi+(x_\alpha-y_\alpha)^*y_\alpha\xi\|\|\xi\|\\
																	&=&\|(x_\alpha-y_\alpha)^*(z_\alpha-y_\alpha)\xi\|\|\xi\|
																	 \leq \|x_\alpha-y_\alpha\|\|z_\alpha\xi-y_\alpha\xi\|\|\xi\|,
\end{eqnarray*}
where we have used the equality $y_\alpha^*y_\alpha=x_\alpha^*y_\alpha$ which follows directly from $y_\alpha\preceq x_\alpha$.
Moreover, since $x_\alpha-y_\alpha\preceq z_\alpha\preceq z_{\alpha_0}$ for all $\alpha\in\Lambda$, we obtain from Lemma~\ref{norm inequality} that $\|x_\alpha-y_\alpha\|\leq \|z_\alpha\|\leq\|z_{\alpha_0}\|$ for all $\alpha\in\Lambda$. Applying what we have just shown,
		$$\|x_\alpha\xi-y_\alpha\xi\| \leq \|x_\alpha-y_\alpha\|^{\frac{1}{2}}\|z_\alpha\xi-y_\alpha\xi\|^{\frac{1}{2}}\|\xi\|^{\frac{1}{2}}\leq \|z_{\alpha_0}\|^{\frac{1}{2}}\|z_\alpha\xi-y_\alpha\xi\|^{\frac{1}{2}}\|\xi\|^{\frac{1}{2}}\to 0.
		$$
Accordingly, $(x_\alpha)_{\alpha\in\Lambda}$ converges strongly to $x$, whence $(x_\alpha)_{\alpha\in\Gamma}$ converges strongly to $x$.

		The inclusion $\tau_s|_{\Mcal}\subseteq\tau_o(\Mcal,\preceq)$ is an immediate consequence of the statement just proved.
		\end{proof}

The fact that the order topology $\tau_o(\Mcal,\preceq)$ on a \vNa{} \Mcal{} is finer than the relative strong topology on \Mcal{} immediately implies that $\tau_o(\Mcal,\preceq)$ is Hausdorff.

\begin{lem}\label{involution}
The involution on a \vNa{} \Mcal{} is order continuous (i.e. $x_\alpha^*\stackrel{o}{\to} x^*$ in $(\Mcal,\preceq)$ whenever $x_\alpha\stackrel{o}{\to} x$ in $(\Mcal,\preceq)$).
\end{lem}
	\begin{proof}
Let $x_\alpha\stackrel{o}{\to} x$ in $(\Mcal,\preceq)$. This means that there are nets $(y_\alpha)_{\alpha\in\Gamma}$ and $(z_\alpha)_{\alpha\in\Gamma}$ in $(\Mcal,\preceq)$ such that $y_\alpha\preceq x_\alpha\preceq z_\alpha$ for all $\alpha\in\Gamma$, $y_\alpha\uparrow x$, and $z_\alpha\downarrow x$. Since the involution preserves the star order, we have $y_\alpha^*\preceq x_\alpha^*\preceq z_\alpha^*$ for all $\alpha\in\Gamma$, $y_\alpha^*\uparrow x^*$, and $z_\alpha^*\downarrow x^*$. It follows from definition of order convergence that $x_\alpha^*\stackrel{o}{\to} x^*$.
	\end{proof}

We have seen in Proposition~\ref{sot} that the (star) order topology is finer than relative strong topology. We observe, by Lemma~\ref{involution}, that if $(x_\alpha)_{\alpha\in\Gamma}$ is order convergent to $x$, then $(x_\alpha)_{\alpha\in\Gamma}$ and $(x_\alpha^*)_{\alpha\in\Gamma}$ are $\tau_s$-convergent to $x$ and $x^*$, respectively. Using this very restrictive (the involution is not continuous in $\tau_s$) necessary condition for order convergence in $(\Mcal,\preceq)$, we obtain a stronger result than Proposition~\ref{sot}. We prove that $\tau_o(\Mcal,\preceq)$ is finer than $\sigma$-strong* topology $s^*(\Mcal,\Mcal_*)$.

\begin{theo}
Let $(x_\alpha)_{\alpha\in\Gamma}$ be a net in a \vNa{} $\Mcal$ and let $x\in\Mcal$. If $x_\alpha\stackrel{o}{\to} x$ in $(\Mcal,\preceq)$, then $x_\alpha\stackrel{s^*(\Mcal,\Mcal_*)}{\to} x$. In particular, $s^*(\Mcal,\Mcal_*)\subseteq \tau_o(\Mcal,\preceq)$.
	\end{theo}
		\begin{proof}
Suppose that $x_\alpha\stackrel{o}{\to}x$ in $(\Mcal,\preceq)$. By Lemma~\ref{limsup}, we can assume without loss of generality that $(x_\alpha)_{\alpha\in\Gamma}$ is bounded in $(\Mcal,\preceq)$. Proposition~\ref{sot} yields $x_\alpha\stackrel{\tau_s}{\to}x$. Combining Proposition~\ref{sot} with Lemma~\ref{involution}, we see that $x_\alpha^*\stackrel{\tau_s}{\to}x^*$. Hence 
$$\left(\|x_\alpha\xi-x\xi\|^2+\|x_\alpha^*\xi-x^*\xi\|^2\right)^\frac{1}{2}\to 0$$
for all $\xi\in\Hscr$, where \Hscr{} is the underlying Hilbert space. Thus $x_\alpha\stackrel{\tau_{s^*}}{\to}x$.

According to Lemma~\ref{norm inequality}, the net $(x_\alpha)_{\alpha\in\Gamma}$ is norm bounded. Moreover, it is well known that topologies $\tau_{s^*}$ and $s^*(\Mcal,\Mcal_*)$ coincide on every norm bounded subset of $\Mcal$. Hence $x_\alpha\stackrel{s^*(\Mcal,\Mcal_*)}{\to} x$. 

The fact $s^*(\Mcal,\Mcal_*)\subseteq \tau_o(\Mcal,\preceq)$ follows directly from what we have just proved.
		\end{proof}

\begin{pro}\label{constant net}
Let $x$ and $y$ be elements of a \vNa{} \Mcal.
%Let $\Mcal$ be a \vNa{}.
	%\begin{enumerate}
		%\item 
		If $x$ is invertible and $x\preceq y$, then $x=y$. Consequently, every order convergent net in $\Mcal_{inv}$ is constant.
		%If $(x_\alpha)_{\alpha\in\Gamma}$ is a net in $\Mcal_{inv}$ which is order convergent to $x\in\Mcal$ in $(\Mcal,\preceq)$, then $x_\alpha=x$ for all $\alpha\in\Gamma$.
		%\item The set $\Mcal_{inv}$ is closed in $\tau_o(\Mcal,\preceq)$.
		%\item Topology $\tau_o(\Mcal_{inv},\preceq)$ is discrete and $\tau_o(\Mcal_{inv},\preceq)=\tau_o(\Mcal,\preceq)|_{\Mcal_{inv}}$.
	%\end{enumerate}
\end{pro}
	\begin{proof}
It follows directly from the definition of the star order that $x=y$ whenever $x$ is invertible and $x\preceq y$.
	
Let $(x_\alpha)_{\alpha\in\Gamma}$ be an order convergent net of invertible elements of \Mcal{}.
 %and let $x\in\Mcal$. Assume that $x_\alpha\stackrel{o}{\to} x$ in $(\Mcal,\preceq)$. 
Then there is a decreasing net $(z_{\alpha})_{\alpha\in\Gamma}$ in $(\Mcal,\preceq)$ such that 
%$z_\alpha\downarrow x$ and 
$x_\alpha\preceq z_\alpha$ for all $\alpha\in \Gamma$. The 
%condition $x_\alpha\preceq z_\alpha$ for all $\alpha\in \Gamma$, together with 
invertibility of elements $x_\alpha$ ensures that $x_\alpha=z_\alpha$ for all $\alpha\in\Gamma$. Therefore, $(x_\alpha)_{\alpha\in\Gamma}$ is decreasing in $(\Mcal,\preceq)$.
 %and $\inf x_\alpha=x$. 
Let $\alpha,\beta\in\Gamma$ be arbitrary. Then there is $\gamma\in\Gamma$ such that $\alpha,\beta\leq \gamma$. 
%By the fact that $(x_\alpha)_{\alpha\in\Gamma}$ is decreasing in $(\Mcal,\preceq)$, 
Hence $x_\gamma\preceq x_\alpha,x_\beta$ and so $x_\alpha=x_\gamma=x_\beta$ because of invertibility of $x_\gamma$. 
%We infer from this and the condition $\inf x_\alpha=x$ that $x_\alpha=x$ for all $\alpha\in\Gamma$.
	\end{proof}

\begin{cor}
Let $\Mcal$ be a \vNa.
	\begin{enumerate}
		%\item If $(x_\alpha)_{\alpha\in\Gamma}$ is a net in $\Mcal_{inv}$ which is order convergent to $x\in\Mcal$ in $(\Mcal,\preceq)$, then $x_\alpha=x$ for all $\alpha\in\Gamma$.
		\item The set $\Mcal_{inv}$ is closed in $\tau_o(\Mcal,\preceq)$.
		\item Topology $\tau_o(\Mcal_{inv},\preceq)$ is discrete and $\tau_o(\Mcal_{inv},\preceq)=\tau_o(\Mcal,\preceq)|_{\Mcal_{inv}}$.
	\end{enumerate}
\end{cor}
	\begin{proof}\hfill{}
	\begin{enumerate}
		%\item Let $(x_\alpha)_{\alpha\in\Gamma}$ be a net of invertible elements of \Mcal{} and let $x\in\Mcal$. Assume that $x_\alpha\stackrel{o}{\to} x$ in $(\Mcal,\preceq)$. Then there is a net $(z_{\alpha})_{\alpha\in\Gamma}$ in $(\Mcal,\preceq)$ such that $z_\alpha\downarrow x$ and $x_\alpha\preceq z_\alpha$ for all $\alpha\in \Gamma$. The condition $x_\alpha\preceq z_\alpha$ for all $\alpha\in \Gamma$, together with invertibility of elements $x_\alpha$, clearly ensures that $x_\alpha=z_\alpha$ for all $\alpha\in\Gamma$. Therefore, $(x_\alpha)_{\alpha\in\Gamma}$ is decreasing in $(\Mcal,\preceq)$ and $\inf x_\alpha=x$. Let $\alpha,\beta\in\Gamma$ be arbitrary. Then there is $\gamma\in\Gamma$ such that $\alpha,\beta\leq \gamma$. By the fact that $(x_\alpha)_{\alpha\in\Gamma}$ is decreasing in $(\Mcal,\preceq)$, $x_\gamma\preceq x_\alpha,x_\beta$ and so $x_\alpha=x_\gamma=x_\beta$ because of invertibility of $x_\gamma$. We infer from this and the condition $\inf x_\alpha=x$ that $x_\alpha=x$ for all $\alpha\in\Gamma$.
		
		\item The fact that $\Mcal_{inv}$ is closed in $\tau_o(\Mcal,\preceq)$ is a direct consequence of Proposition~\ref{constant net}.
		
		\item If $M\subseteq\Mcal_{inv}$, then $M$ is closed in $\tau_o(\Mcal_{inv},\preceq)$ because of Proposition~\ref{constant net}. This proves that $\tau_o(\Mcal_{inv},\preceq)$ is discrete.
		
		Every nonempty subset of $\Mcal_{inv}$ which has an upper bound in $(\Mcal,\preceq)$ contains only one element. Therefore, the supremum of every nonempty subset of $\Mcal_{inv}$ with an upper bound in $(\Mcal,\preceq)$ belongs to $\Mcal_{inv}$. Combining (i), Corollary~\ref{supinf1}(i), and Proposition~\ref{Dedekind}, we obtain $\tau_o(\Mcal_{inv},\preceq)=\tau_o(\Mcal,\preceq)|_{\Mcal_{inv}}$.
	\end{enumerate}
	\end{proof}

\begin{cor}\label{comparison}
		The norm topology $\tau_u$ on a \vNa{} \Mcal{} is not finer than $\tau_o(\Mcal,\preceq)$.
\end{cor}
	\begin{proof}
Consider the set $M=\{\frac{1}{n}\unit|\,n\in\nn\}$. Since $M$ is a set of invertible elements, it is closed in $\tau_o(\Mcal,\preceq)$. However, $M$ is not closed in $\tau_u$.
	\end{proof}

We have seen that the norm topology is not finer than the order topology. Now, let us concentrate on the converse question whether the order topology is finer than the norm topology.

\begin{lem}\label{dimension}
Let \Mcal{} be a \vNa{}. The following statements are equivalent:
	\begin{enumerate}
		\item $\Mcal$ admits no infinite family $(p_\alpha)_{\alpha\in I}$ of mutually orthogonal nonzero projections with $\sup_{\alpha\in I} p_\alpha=\unit$.
		\item $\Mcal$ is finite-dimensional.
		\item $\Mcal$ is (isomorphic to) a finite direct sum of full matrix algebras.
	\end{enumerate}
\end{lem}
	\begin{proof}
From \cite[Exercise~5.7.39]{Ka91}, we have $(i)\Rightarrow(ii)$. It follows from \cite[Proposition~6.6.6]{Ka97II} and \cite[Theorem~6.6.1]{Ka97II} that $(ii)\Rightarrow(iii)$. The statement $(iii)\Rightarrow(i)$ is clear.
	\end{proof}

\begin{theo}
If a \vNa{} \Mcal{} is infinite-dimensional, then the set $\Mcal\setminus\Mcal_{inv}$ of all noninvertible elements in \Mcal{} is not order closed. In this case, the topology $\tau_o(\Mcal,\preceq)$ is not comparable with the norm topology $\tau_u(\Mcal)$.
\end{theo}
	\begin{proof}
It follows from Lemma~\ref{dimension}	that there is an infinite family $(p_\alpha)_{\alpha\in I}$ of mutually orthogonal nonzero projections in \Mcal{} satisfying $\sup_{\alpha\in I} p_\alpha=\unit$.
The set $\Gamma$ consisting of all finite subsets of $I$ is directed by the inclusion relation. Consider the net $(x_F)_{F\in\Gamma}$ of projections 
$$
x_F=\sup_{\alpha\in F} p_\alpha=\sum_{\alpha\in F}p_\alpha. 
$$
It is easy to see that $(x_F)_{F\in\Gamma}$ is increasing. Moreover, if $F\in \Gamma$ and $\beta\in I\setminus F$, then 
$$
p_\beta x_F=p_\beta\sum_{\alpha\in F}p_\alpha=\sum_{\alpha\in F}p_\beta p_\alpha=0.
$$
Thus $x_F$ is not invertible for each $F\in\Gamma$. Furthermore, $x_F\preceq x_F\preceq\unit$ for every $F\in\Gamma$ and $\sup_{F\in\Gamma}x_F=\unit$. This shows that the net $(x_F)_{F\in\Gamma}$ of noninvertible projections order converges to $\unit$ in $(\Mcal,\preceq)$. Hence $\Mcal\setminus\Mcal_{inv}$ is not order closed in $(\Mcal,\preceq)$.

It remains to show that $\tau_o(\Mcal,\preceq)$ is not comparable with the norm topology $\tau_u(\Mcal)$. It follows from what we have proved above that $\tau_o(\Mcal,\preceq)$ is not finer than the norm topology $\tau_u(\Mcal)$ because the set $\Mcal\setminus\Mcal_{inv}$ is closed in the norm topology \cite[Proposition~3.1.6]{Ka97I}. In addition, by Corollary~\ref{comparison}, $\tau_u(\Mcal)$ is not finer than $\tau_o(\Mcal,\preceq)$.
	\end{proof}

In order to complete our discussion about comparison of the order topology $\tau_o(\Mcal,\preceq)$ on a \vNa{} \Mcal{} with the norm topology, we shall prove that if \Mcal{} is finite-dimensional, then the order topology $\tau_o(\Mcal,\preceq)$ is necessarily discrete and so it is strictly finer than the norm topology. 

\begin{theo}
If a \vNa{} \Mcal{} is finite-dimensional, then the order topology $\tau_o(\Mcal,\preceq)$ is discrete.
\end{theo}
	\begin{proof}
Since \Mcal{} is finite-dimensional, we see from Lemma~\ref{dimension} that there is no infinite family of mutually orthogonal nonzero projections. Then every projection in \Mcal{} has only a finite number of mutually orthogonal nonzero subprojections. 

We now prove that every increasing net of projections in $(\Mcal,\preceq)$ is eventually constant. Let $(p_\alpha)_{\alpha\in\Gamma}$ be an increasing net of projections $(\Mcal,\preceq)$. Suppose that $(p_\alpha)_{\alpha\in\Gamma}$ is not eventually constant. Then there is $\alpha_0\in\Gamma$ such that $p_{\alpha_0}\neq 0$. Since $(p_\alpha)_{\alpha\in\Gamma}$ is increasing and is not eventually constant, there is $\alpha_1\in\Gamma$ such that $\alpha_0\leq \alpha_1$ and $p_{\alpha_0}< p_{\alpha_1}$. Proceeding by induction, we obtain an increasing sequence $(\alpha_n)_{n\in\nn_0}$ in $\Gamma$ such that $p_{\alpha_m}<p_{\alpha_n}$ whenever $m,n\in\nn_0$ satisfy $m<n$. Set $e_0=p_{\alpha_0}$ and $e_{n+1}=p_{\alpha_{n+1}}-p_{\alpha_n}$ for all $n\in\nn_0$. Clearly, $(e_n)_{n\in\nn_0}$ is a sequence of mutually orthogonal nonzero projections. Thus the projection $\sup_{n\in\nn_0}e_n$ in \Mcal{} has infinite number of mutually orthogonal nonzero subprojections which is a contradiction. This proves that every increasing net of projections in $(\Mcal,\preceq)$ is eventually constant.

Let us show that every decreasing or increasing net in $(\Mcal,\preceq)$ is necessarily eventually constant. Assume that $(x_\alpha)_{\alpha\in \Gamma}$ is an increasing net in $(\Mcal,\preceq)$. By Proposition~\ref{antezana}, $(l(x_\alpha))_{\alpha\in\Gamma}$ is an increasing net of projections in $(\Mcal,\preceq)$ and so it is eventually constant. This means that there is $\alpha_0\in\Gamma$ such that $l(x_\alpha)=l(x_{\alpha_0})$ whenever $\alpha\in\Gamma$ is such that $\alpha_0\leq \alpha$. Employing Proposition~\ref{antezana}, $$
x_{\alpha_0}=l(x_{\alpha_0})x_\alpha=l(x_\alpha)x_\alpha=x_\alpha
$$ 
for every $\alpha\in \Gamma$ satisfying $\alpha_0\leq \alpha$. Now suppose that $(x_\alpha)_{\alpha\in \Gamma}$ is a decreasing net in $(\Mcal,\preceq)$. By Proposition~\ref{antezana}, $(\unit-l(x_\alpha))_{\alpha\in\Gamma}$ is an increasing net of projections in $(\Mcal,\preceq)$. Hence $(\unit-l(x_\alpha))_{\alpha\in\Gamma}$ is eventually constant which implies that $(l(x_\alpha))_{\alpha\in\Gamma}$ is eventually constant. Now it follows from a similar argument as in the case of an increasing net that $(x_\alpha)_{\alpha\in\Gamma}$ is eventually constant.

Let a net $(x_\alpha)_{\alpha\in\Gamma}$ in $M\subseteq \Mcal$ be order convergent to $x$ in $(\Mcal,\preceq)$. Then there are nets $(y_\alpha)_{\alpha\in\Gamma}$ and $(z_\alpha)_{\alpha\in\Gamma}$ in $(\Mcal,\preceq)$ such that $y_\alpha\preceq x_\alpha\preceq z_\alpha$ for every $\alpha\in\Gamma$, $y_\alpha\uparrow x$, and $z_\alpha\downarrow x$. By the previous part of the proof, $(y_\alpha)_{\alpha\in\Gamma}$ and $(z_\alpha)_{\alpha\in\Gamma}$ are eventually constant. Hence there is $\beta\in\Gamma$ such that $(y_\alpha)_{\alpha\in\Lambda}$ and $(z_\alpha)_{\alpha\in\Lambda}$, where $\Lambda=\{\alpha\in\Gamma|\,\beta\leq\alpha\}$, are constant nets. It follows from the arguments used in the proof of Lemma~\ref{limsup} that the supremum of $(y_\alpha)_{\alpha\in\Lambda}$ and the infimum of $(z_\alpha)_{\alpha\in\Lambda}$ are equal to $x$. We infer from this that $y_\beta=z_\beta=x$. Accordingly, $x_\beta=x$ because 
$$
x=y_\beta\preceq x_\beta\preceq z_\beta=x.
$$
We have proved that $x$ has to be an element of $M$. Thus every subset of \Mcal{} is order closed and so $\tau_o(\Mcal,\preceq)$ is discrete.
	\end{proof}

At the end of this section, we discuss relationships between topologies $\tau_o(\Mcal,\preceq)$, $\tau_o(\Mcal_{pi},\preceq)$, $\tau_o(\Mcal_{sa},\preceq)$, $\tau_o(\Mcal_+,\preceq)$, and $\tau_o(P(\Mcal),\preceq)$. We shall see in Corollary~\ref{relationships} that a relation between $\tau_o(\Mcal_{sa},\preceq)$ and $\tau_o(P(\Mcal),\preceq)$ is very different from order topologies generated by the standard order (see \cite[Proposition 2.9]{ChHW15}).

\begin{pro}
Let \Mcal{} be a \vNa. The sets $P(\Mcal)$, $\Mcal_+$, $\Mcal_{pi}$, and $\Mcal_{sa}$ are closed in $\tau_o(\Mcal,\preceq)$.
\end{pro}
\begin{proof}
As $P(\Mcal)$, $\Mcal_+$, and $\Mcal_{sa}$ are strongly operator closed, it follows from Proposition~\ref{sot} that they are closed in $\tau_o(\Mcal,\preceq)$.

Assume that $(x_\alpha)_{\alpha\in\Gamma}$ be a net of partial isometries such that $x_\alpha\stackrel{o}{\to} x\in\Mcal$ in $(\Mcal,\preceq)$. Then there is a net $(y_\alpha)_{\alpha\in\Gamma}$ satisfying $y_\alpha\preceq x_\alpha$ for all $\alpha\in\Gamma$ and $y_\alpha\uparrow x$. By Lemma~\ref{smaller element}, $(y_\alpha)_{\alpha\in\Gamma}$ is a net of partial isometries. According to Corollary~\ref{supinf2}, $x$ is a partial isometry in \Mcal{}. Thus $\Mcal_{pi}$ is closed in $\tau_o(\Mcal,\preceq)$.
\end{proof}

\begin{cor}\label{relationships}
Let \Mcal{} be a \vNa. Then 
	\begin{enumerate}
		\item $\tau_o(\Mcal,\preceq)|_{\Mcal_{pi}}=\tau_o(\Mcal_{pi},\preceq)$;
		\item $\tau_o(\Mcal,\preceq)|_{\Mcal_{sa}}=\tau_o(\Mcal_{sa},\preceq)$;
		\item $\tau_o(\Mcal,\preceq)|_{\Mcal_+}=\tau_o(\Mcal_{sa},\preceq)|_{\Mcal_+}=\tau_o(\Mcal_+,\preceq)$;
		\item $\tau_o(\Mcal,\preceq)|_{P(\Mcal)}=\tau_o(\Mcal_{sa},\preceq)|_{P(\Mcal)}=\tau_o(\Mcal_+,\preceq)|_{P(\Mcal)}=\tau_o(P(\Mcal),\preceq)$;
	\end{enumerate}
\end{cor}
	\begin{proof}
The statements (i)-(iv) follow directly from Proposition~\ref{Dedekind}, Corollary~\ref{supinf1}, Corollary~\ref{supinf2}, and the previous proposition.	
	\end{proof}

%%%%%%%%%%%%%%%%%%%%%%%%%%%%%%%%%%%%%%%%%%%%%%%%%%%%%%%%%%%%%%%%%%%%%%%%
%%%%%%%%%%%%%%%%%%%%%%%%%%%%%%%%%%%%%%%%%%%%%%%%%%%%%%%%%%%%%%%%%%%%%%%%

\section*{Acknowledgement}
This work was supported by  the ``Grant Agency of the Czech Republic" grant number 17-00941S, ``Topological and geometrical properties of Banach spaces and operator algebras II". 

%%%%%%%%%%%%%%%%%%%%%%%%%%%%%%%%%%%%%%%%%%%%%%%%%%%%%%%%%%%%%%%%%%%%%%%%


\begin{thebibliography}{99}

\bibitem{ACMS10} J. Antezana, C. Cano, I. Musconi, and D. Stojanoff: {\em A note on the star order in Hilbert spaces,} Linear Multilinear Algebra {\bf 58} (2010), 1037-1051.

\bibitem{BHLL04} J. K. Baksalary, J. Hauke, X. Liu, and S. Liu: {\em Relationships between partial orders of matrices and their powers}, Linear Algebra Appl. {\bf 379} (2004), 277--287. 

\bibitem{Bi35} G. Birkhoff: {\em On the structure of abstract algebras}, Proc.  Camb. Philos. Soc. {\bf 31} (1935), 433--454. 

\bibitem{Bi67} G. Birkhoff: {\em Lattice Theory}, American Mathematical Society, New York (1967). 

\bibitem{Bo11} M. Bohata: {\em Star order on operator and function algebras}, Publ. Math. Debrecen {\bf 79} (2011), 211--229. 

\bibitem{BChW12} 
D. Buhagiar, E. Chetcuti, and H. Weber: {\it The order topology on the projection lattice of a Hilbert space}, Topol. Appl. {\bf 159} (2012), 2280--2289.

\bibitem{ChHW15} 
E. Chetcuti, J. Hamhalter, and H. Weber: {\it The order topology for a \vNa{}}, Studia Math. {\bf 230} (2015), 95--120.

\bibitem{Ci15} 
J. C\={\i}rulis: {\it Lattice operations on Rickart *-rings under the star order}, Linear Multilinear Algebra {\bf 63} (2015), 497--508.

\bibitem{Dr78}
M. P. Drazin: \textit{Natural structures on semigroups with involution}, Bull. Amer. Math. Soc. {\bf 84} (1978), 139--141.

\bibitem{ER95} 
M. Ern\'e and Z. Rie\v{c}anov\'a: {\it Order-topological complete orthomodular lattices}, Topol. Appl. {\bf 61} (1995), 215--227.

\bibitem{Ex17} 
R. Exel: {\it Partial Dynamical Systems, Fell bundles and Applications}, Amer. Math. Soc., Providence, 2017.

\bibitem{FK54} 
E. E. Floyd and V.L. Klee: {\it A characterization of reflexivity by the lattice of closed subspaces}, Proc. Amer. Math. Soc. {\bf 5} (1954), 655--661.

\bibitem{Gi76} 
A. R. Gingras: \textit{Convergence lattices}, Rocky Mountain J. Math. {\bf 6} (1976), 85--104.

\bibitem{Gu06} 
S. Gudder: \textit{An order for quantum observables}, Math. Slovaca {\bf 56} (2006), 573--589.

\bibitem{HL63} 
P. R. Halmos and J. E. McLaughlin: \textit{Partial isometries}, Pacific J. Math. {\bf 13} (1963), 585--596.

\bibitem{Ja83} 
M. F. Janowitz: \textit{On the *-order for Rickart *-rings}, Algebra Universalis {\bf 16} (1983), 360--369.

\bibitem{Ka97I} 
R. V. Kadison and J. R. Ringrose: \textit{Fundamentals of the Theory of Operator Algebras I}, Amer. Math. Soc., Providence, 1997. 

\bibitem{Ka97II} 
R. V. Kadison and J. R. Ringrose: \textit{Fundamentals of the Theory of Operator Algebras II}, Amer. Math. Soc., Providence, 1997. 

\bibitem{Ka91} 
R. V. Kadison and J. R. Ringrose: \textit{Fundamentals of the Theory of Operator Algebras III}, Amer. Math. Soc., Providence, 1991. 

\bibitem{Pa95} 
V. Palko: \textit{The weak convergence of unit vectors to zero in Hilbert space is the convergence of one-dimensional subspaces in the order topology}, Proc. Amer. Math. Soc. {\bf 123} (1995), 715--721. 

\bibitem{PV07} 
S. Pulmannov\'a and E. Vincekov\'a: \textit{Remarks on the order for quantum observables}, Math. Slovaca {\bf 57} (2007), 589--600.

\end{thebibliography}
\end{document}